\documentclass{amsart}

\usepackage{amssymb}
\usepackage[stretch=10,shrink=10]{microtype}
\usepackage[shortlabels]{enumitem}
\usepackage{tikz}
\usetikzlibrary{arrows.meta}
\usepackage{comment}
\usepackage{subfigure}
\usepackage{pgfplots}
\usepackage{hyperref}
\hypersetup{colorlinks,linkcolor={red},citecolor={olive},urlcolor={red}}
\usepackage{mathtools}
\mathtoolsset{showonlyrefs}

\newcommand{\f}{\frac}

\renewcommand{\phi}{\varphi}
\newcommand{\E}{\mathbf E}

\newcommand{\A}{\mathcal{A}}

\renewcommand{\P}{\mathbf P}

\newcommand{\tp}{\tilde p}
\newcommand{\tU}{\tilde U}

\DeclareMathOperator{\Poi}{Poi}

\DeclareMathOperator{\SFM}{SFM}
\DeclareMathOperator{\FM}{FM}

\DeclareMathOperator{\NBFM}{nbFM}

\DeclareMathOperator{\NBBRW}{nbBRW}
\DeclareMathOperator{\RNBBRW}{rnbBRW}

\usepackage{theoremref}

\newtheorem{theorem}{Theorem}
\newtheorem{lemma}[theorem]{Lemma}
\newtheorem{fact}[theorem]{Fact}

\newtheorem{proposition}[theorem]{Proposition}

\theoremstyle{definition}

\newtheorem{remark}[theorem]{Remark}

    \title{Critical drift estimates for the frog model on trees}

\author{Emma Bailey} \email{ebailey@gc.cuny.edu}

\author{Matthew Junge} \email{Matthew.Junge@baruch.cuny.edu}

\author{Jiaqi Liu} \email{liujiaqi@sas.upenn.edu}

\thanks{All authors were partially supported by NSF grant \#2115936.
}

\begin{document}
\maketitle

\begin{abstract}
    Place an active particle at the root of a $d$-ary tree and a single dormant particle at each non-root site. In discrete time, active particles move towards the root with probability $p$ and, otherwise, away from the root to a uniformly sampled child vertex.  When an active particle moves to a site containing a dormant particle, the dormant particle becomes active. The critical drift $p_d$ is the infimum over all $p$ for which infinitely many particles visit the root almost surely. Guo, Tang, and Wei proved that $\sup_{d\geq 3} p_d \leq 1/3$. We improve this bound to $5/17$ with a shorter argument that generalizes to give bounds on $\sup_{d \geq m} p_d$. We additionally prove that $\limsup p_d \leq 1/6$ by finding the limiting critical drift for a non-backtracking variant.
\end{abstract}

\section{Introduction}



Let $\mathbb T_d$ be the infinite, rooted $d$-ary tree in which each vertex has $d\geq 2$ child vertices. Place an active particle at the root $\varnothing$ and a single dormant particle at each non-root site. Fix $p \in (0,1)$
and have each active particle perform a nearest neighbor $p$-biased random walk. At each discrete time step an active particle moves one step towards the root with probability $p$ and away from the root to a uniformly sampled child vertex with probability $1-p$. 
When an active particle moves to a site with a dormant particle, the dormant particle becomes active and begins its own independent $p$-biased random walk. Call this process the \emph{frog model with drift on $\mathbb T_d$} and denote it by $\FM(d,p)$.

Frog model dynamics capture aspects of the spread of infection, a rumor, or energy. When the underlying graph is infinite, a basic question is whether or not infinitely many particles visit the root. Many papers have studied this with simple random walks on integer lattices and trees \cite{telcs1999branching, alves2002phase, popov2001frogs, hoffman2017recurrence, hoffman2019infection, michelen2019frog}. There has been recent interest in the variant in which active particles perform biased random walk \cite{dobler2018reccurence, beckman2019frog, guo2022minimal}.

A \emph{root visit} is counted each time an active particle moves to the root. Let $V_t$ be the number of root visits up to time $t$ and $V_{\FM(d,p)}\coloneqq \lim_{t \to \infty} V_t$ be the total number of root visits. We say that  $\FM(d,p)$ is \emph{recurrent} if $\P(V_{\FM(d,p)}=\infty) =1 $. Recurrence satisfies a 0-1 law (see \cite[Proof of Proposition 1.4]{beckman2019frog}). Accordingly, call the process \emph{transient} if it is not recurrent. Define 
\begin{align}
    p_d \coloneqq \inf \{ p \colon \FM(d,p) \text{ is recurrent}\}
\end{align}
to be the infimum over all drifts for which $\FM(d,p)$ is recurrent. Since a single $p$-biased walk is recurrent for $p\geq 1/2$, we are only interested in $p\in (0,1/2)$.

A natural case $\FM(2,1/3)$ has active particles performing simple random walk on the binary tree. Hoffman, Johnson, and Junge resolved a longstanding open problem by proving that $\FM(2,1/3)$ is recurrent \cite{hoffman2017recurrence}. Conversely, $\FM(2,p)$ with $p<1/3$ is transient since the dominating process with all particles initially active is transient. Thus, $p_2 =1/3$. 

Increasing $p$ creates a stronger drift towards the root, and increasing $d$ results in more dormant frogs. Both of these effects should result in more visits to the root. An intriguing feature of $\FM(d,p)$ is that there is no known proof that $V_{\FM(d,p)}$ stochastically increases in $p$ or $d$. For example, it is not obvious that $p_2=1/3$ is a uniform bound for $p_d$. The main result from \cite{beckman2019frog} proved a weaker bound $\sup_{d \geq 3} p_d \leq 0.4155$. In \cite{guo2022minimal}, Guo, Tang, and Wei improved this bound to $1/3$. Our first result is a slightly better bound.

\begin{theorem} \thlabel{thm:sup}
    $\sup_{d \geq 3} p_d \leq 5/17 \approx 0.2941.$
\end{theorem}

 Besides the bound improvement, we see several positive consequences of \thref{thm:sup}. One is that the proof uses a different technique than what was used in \cite{guo2022minimal}, which provides new perspective. A particular highlight is a simple to check criteria for recurrence of $\FM(d,p)$ in \thref{prop:suff}. Another nice consequence is that our proof is shorter than that given in \cite{guo2022minimal} (four versus nineteen pages). The third is that the bound we obtain is strictly less than $p_2=1/3$. This implies that $p_d < p_2$ for all $d \geq 3$, which supports the conjecture from \cite{beckman2019frog} and \cite{guo2022minimal} that $p_d$ is decreasing. The fourth positive consequence is that our technique can be generalized to give better bounds on $\sup_{d \geq m} p_d$ as $m$ is increased. See \thref{rem:m} for more details. It is unclear if the approach from \cite{guo2022minimal} could be as easily generalized.  To illustrate how the generalization goes, we prove an extension for $d=4$. 

\begin{theorem}\thlabel{cor:4}
    $\sup_{d \geq 4} p_d \leq 27/100.$
\end{theorem}


The authors of \cite{beckman2019frog} further conjectured that 
\begin{align}
\lim_{d \to \infty} p_d = q^*\coloneqq\frac{2- \sqrt 2}{4} \approx 0.1464. \label{eq:lim}
\end{align}
Here $q^*$ is the critical drift for the branching $p$-biased random walk in which each particle does not branch when moving towards the root (which it does with probability $p$) and splits into two particles when moving away from the root. Our second result is a limiting bound on $p_d$ that is near $q^*$.

\begin{theorem} \thlabel{thm:limsup}
$\limsup_{d \to \infty}  p_d \leq 1/6 \approx 0.1667.$
\end{theorem}

Intuition suggests that as $d$ becomes larger, most steps away from the root by particles in $\FM(d,p)$ will be to sites containing dormant particles. Thus, $\FM(d,p)$ ought to converge to this branching random walk as $d \to \infty$. As mentioned previously, monotonicity of $p_d$ has yet to be established. So, both the existence of the limit and convergence to $q^*$ remain open. 

The only known monotonicity result for $\FM(d,p)$ is \cite[Proposition 1.2]{beckman2019frog}, which states that $V_{\FM(d,p)} \preceq V_{\FM(kd,p)}$ for any positive integer $k$. 
One difficulty is that the frog model has regimes in which the set of sites visited by active particles contains a linearly expanding ball centered at the root \cite{hoffman2019infection}. No activation occurs in this growing region. This distinguishes the frog model from branching random walk on a macroscopic level, and casts a shadow of doubt on \eqref{eq:lim}.


\thref{thm:limsup} is proven by exactly computing the limiting critical drift for the {non-backtracking frog model} denoted by $\NBFM(d,p)$. 
This is a relevant model since all arguments that we know of for recurrence of $\FM(d,p)$ rely on proving that $\NBFM(d,p)$ is recurrent. 
In $\NBFM(d,p)$, paths of active particles are non-backtracking. Let
\begin{align}
	p^* = p^*(p,d) \coloneqq \f{ p(d-1)}{d - (d+1) p} \text{ and } \hat p = \hat p(p) \coloneqq \f{p}{1-p}. \label{eq:p*}
\end{align}  
Initially, there is one active frog at the root. It moves to a uniformly sampled child vertex in the first step and activates the dormant frog there. Just activated frogs move towards the root with probability $p^*$, and otherwise away from the root to a uniformly sampled child vertex. For subsequent steps, if the previous step was towards the root, then the next step will be towards the root with probability $\hat p$. If the previous step was away from the root, all subsequent steps will be away from the root to uniformly sampled child vertices. Any particles that visit the root are killed there and no longer participate in the process. Let $V_{\NBFM(d,p)}$ denote the total number of root visits in $\NBFM(d,p)$ and say that the process is recurrent if $\P(V_{\NBFM(d,p)} = \infty) =1$. Define $$p_d' = \inf \{ p \colon  \text{$\NBFM(d,p)$ is recurrent} \}.$$ 
We find the exact limiting value of $p_d'$.

\begin{theorem}\thlabel{thm:lim}
$\lim_{d \to \infty} p_d' = 1/6$.
\end{theorem}

\thref{thm:lim} is used to derive \thref{thm:limsup}. \thref{thm:lim} is a valuable contribution in and of itself since it suggests the truth of \eqref{eq:lim}. Indeed, the intuitive limit $\NBFM(\infty ,p)$ is a branching process $\NBBRW(p)$ in which particles move towards the root with probability $\hat p$ and do not branch for some geometric distributed number of steps, after which they move away from the root branching into two particles at each step. It follows from \thref{lem:nbBRW} that $\NBBRW(p)$ has critical drift $1/6$, thus $p_d'$ converges to its intuitive limit. Note that $\NBFM(p)$ also exhbits a linearly expanding ball of visited sites when the initial particle density is high enough \cite{hoffman2019infection}. So, the ``shadow of doubt" mentioned earlier from this macroscopic effect does not seem to effect convergence of the critical drift.

Another benefit of of \thref{thm:lim} is that it provides useful guidance on where \emph{not} to direct future efforts towards establishing \eqref{eq:lim}. All proofs that we know of for recurrence of a frog model on an infinite tree did so by proving that a non-backtracking sub-process is recurrent. Since $q^* < 1/6$, our result suggests that any argument using a non-backtracking frog model will fall short of proving that $p_d \to q^*$. Some new type of argument that engages directly with $\FM(d,p)$ appears to be needed.

The arguments we employ to upper bound $p_d$ and $p_d'$ use approximations to the frog model that are less recurrent. To get a sense of how much precision is lost, we conducted some numerical simulations to estimate $p_3, p_3',$ and $p_4'$. We found that
\begin{align}
p_3\approx0.25;\qquad\qquad p'_3\approx0.2725;\qquad\qquad p'_4\approx0.246.
\label{eq:sims}
\end{align}
Details are in Section~\ref{sec:sims}.  In~\cite{hoffman2017recurrence}, it was conjectured that $\FM(3,\frac{1}{4})$ i.e., the frog model with simple random walks, is recurrent. So, under this assumption $p_3=\frac{1}{4}$.  Our data gives more support to this conjecture (see Figure~\ref{fig:fmd3}). The values of $p'_3$ and $p'_4$ are within about $.02$ of the bounds from \thref{thm:sup} and \thref{cor:4}. This suggests that the  our proofs do not sacrifice much accuracy. It is also interesting to see the (simulated) discrepancy between $p_3$ and $p_3'$ (about $.0225$) that results from restricting to non-backtracking random walk paths. 




\subsection{Overview of proofs}
The proof that $\sup_{d \geq 3}p_d \leq 1/3$ from \cite{guo2022minimal} followed the blueprint from \cite{hoffman2017recurrence}. 
The calculations in \cite{hoffman2017recurrence} were involved, and became much more complex in the generalization in \cite{guo2022minimal}.  We work with a Poisson-distributed number of dormant particles per site. Poisson thinning makes many intricate dependencies vanish. A comparison result from \cite{johnson2018stochastic} lets us convert our findings back to the one particle per site setting of $\FM(d,p)$.

Our main tool is a self-similar frog model $\SFM(d,p)$ that embeds in the usual frog model so that it has fewer root visits. We denote by $V_{\SFM(d,p)}$ the number of root visits in $\SFM(d,p)$. It was observed in \cite{hoffman2016transience} that $V_{\SFM(d,p)}$ satisfies a recursive distributional equation in the simple random walk setting. A similar equation holds for arbitrary $p$. The equation relates $V_{\SFM(d,p)}$ to $1+U$ thinned independent copies of $V_{\SFM(d,p)}$, where $U$ is the number of leaves visited in a frog model on a star graph (see Figure~\ref{fig:Aasystem}). In \thref{prop:suff}, we reduce proving recurrence to finding a stochastic lower bound for $U$ whose Laplace transform satisfies a certain inequality.

\thref{cor:4} uses the same approach as \thref{thm:sup}, and \thref{thm:limsup} follows immediately from \thref{thm:lim}. The proofs of \thref{thm:sup} and \thref{thm:lim} come down to constructing the right stochastic lower bound for $U$. For \thref{thm:sup}, we modify what occurs on $\mathbb T_d$ to resemble the setting with $d=3$. 
For \thref{thm:lim}, we leverage the fact that when the drift is fixed, we do not need many leaves of the star graph to be visited in order to satisfy \thref{prop:suff}.
The arguments presented are not simple rehashes of past techniques. The stochastic lower bounds are novel and tailored to $\FM(d,p)$. 
See \thref{rem:balance} for more about the difficulties.

Another ingredient in the proof of \thref{thm:lim} is connecting $\NBFM(d,p)$ with its intuitive limiting multitype branching random walk. This substantial endeavor is a technical contribution. Section~\ref{sec:NBBRW} defines the multitype branching random walk and then works out its transience and recurrence properties. The main thrust is extending results from \cite{machado2001recurrence} to our setting. The transience/recurrence criteria in \thref{lem:VnbBRW} are novel and may be of future use for the study of multitype branching random walks. 

\subsection{Organization} In Section~\ref{sec:ssfm}, we define the self-similar frog model and deduce some of its properties. This culminates with a sufficient condition for recurrence of $\SFM(d,p)$ given at \thref{prop:suff}. We use this in Section~\ref{sec:sup} to prove \thref{thm:sup}. 
 Section~\ref{sec:NBBRW} gives transience and recurrence conditions for a multitype branching random walk and relate them back to the frog model. In Section~\ref{sec:lim}, we prove \thref{thm:limsup}, which has \thref{thm:lim} as an immediate corollary. Finally, in Section~\ref{sec:sims}, we provide some numerical simulations that complement our results. 

\subsection{Acknowledgements} We are grateful to the authors of~\cite{hoffman2017recurrence} whose code formed the basis of the simulations performed in Section~\ref{sec:sims}. We would also like to thank Serguei Popov for sending us an electronic copy of \cite{comets1998one} whose result is applied in the proof of \thref{{lem:nbBRW}}.

\section{The self-similar frog model and associated operator} \label{sec:ssfm}

First a few remarks on notation. We abbreviate the Poisson distribution with mean $\lambda$ by $\Poi(\lambda)$. Given two nonnegative random variables $X$ and $X'$, we say that $X$ is stochastically smaller than $X'$ if $\P(X \geq a) \leq \P(X' \geq a)$ for all $a \geq 0$. We will denote this by $X \preceq X'$. Similarly, given two probability measures $\pi$ and $\pi'$ on $[0,\infty]$ we say that $\pi \preceq \pi'$ if $\pi((a,\infty)) \leq \pi'(a,\infty)$ for all $a \geq 0$. 



\subsection{The process}

The \emph{self-similar frog model} $\SFM(d, p)$ has particles follow the same type of non-backtracking random walks as in $\NBFM(d,p)$ with some key amendments. The first modification is that we replace the single dormant particle at each site with independent $\Poi(1)$-distributed numbers of particles. When an active particle visits a site with dormant particles, all dormant particles there become active. The additional modification is that particles moving away from the root are killed upon visiting a vertex that has already been visited. If multiple active particles attempt to move away from the root to the same unvisited vertex, then one is chosen to continue its path and the others are killed. Let $V_{\SFM(d,p)}$ denote the total number of root visits in $\SFM(d,p)$. Notice that $\SFM(d,p)$ is defined with a Poisson distributed initial configuration of dormant particles. We use a result from \cite{johnson2018stochastic} to show that this can be compared to prove recurrence of $\FM(d,p)$.

\begin{lemma} \thlabel{lem:so}
    If $\SFM(d,p)$ is recurrent, then $\NBFM(d,p)$ is recurrent. If $\NBFM(d,p)$ is recurrent, then $\FM(d,p)$ is recurrent.
\end{lemma}

\begin{proof}
 \cite[Corollary 5]{johnson2018stochastic} states that recurrence of a frog model with Poisson initial conditions implies recurrence of the same model with one particle per site. The result follows from this and the construction in \cite[Section 2]{guo2022minimal}. The construction explains how $\SFM(d,p)$ is a restriction of $\NBFM(d,p)$, which is a restriction of $\FM(d,p)$. 
\end{proof}
    
\subsection{The operator} \label{sec:A}
\begin{figure}
  \begin{center}
\mbox{
\subfigure
{
\begin{tikzpicture}[scale = .75]
\draw (-5,3) -- node[below] {$\varnothing$} ( -4,3) -- node (root) {} ( -4,4) -- (-5,4) -- cycle; 

\draw[thick, fill = red] (-2.5,2.5) -- node[below] {$\varnothing'$} (-1.5,2.5) -- node (root') {} (-1.5,3.5) -- (-2.5,3.5) -- node (root'2) {} (-2.5,2.5) --cycle;

\draw[fill = blue] (0,5) --node[below] {$v_d$} (1,5) -- (1,6) -- (0,6) -- node (vd) {} (0,5)--cycle;

\node at (.5,4) {$\vdots$};

\draw[fill = blue] (0,2) --node[below] {$v_2$} (1,2) -- (1,3) -- (0,3) -- node (v2) {}  (0,2)--cycle;

\draw[thick,fill = red] (0,0) --node[below] {$v_1$} (1,0) -- (1,1) -- (0,1) --  node (v1) {}  (0,0)--cycle;

\draw[thick] (root') -- (vd);
\draw[thick] (root') -- (v2);
\draw[thick] (root') -- (v1);
\draw[thick] (root'2) -- (root);

\end{tikzpicture}
}

\subfigure{

\begin{tikzpicture}[scale = .75]
\draw[thick] (root'2) -- (root);
\draw (-5,3) -- node[below] {$\varnothing$} ( -4,3) -- node (root) {} ( -4,4) -- (-5,4) -- cycle; 
\draw[fill = red] (-5,3) -- ( -4.7,3) -- node (root) {} ( -4.7,4) -- (-5,4) -- cycle; 
\draw[fill = blue] (-4,3) -- ( -4.7,3) -- node (root) {} ( -4.7,4) -- (-4,4) -- cycle; 

\draw[thick] (-2.5,2.5) -- node[below] {$\varnothing'$} (-1.5,2.5) -- node (root') {} (-1.5,3.5) -- (-2.5,3.5) -- node (root'2) {} (-2.5,2.5) --cycle;

\draw[fill = blue] (0,5) --node[below] {$v_d$} (1,5) -- (1,6)  -- (0,6) -- node (vd) {} (0,5)--cycle;

\node at (.5,4) {$\vdots$};

\draw[thick] (0,2) --node[below] {$v_2$} (1,2) -- (1,3) -- (0,3) -- node (v2) {}  (0,2)--cycle;

\draw[thick] (0,0) --node[below] {$v_1$} (1,0) -- (1,1) -- (0,1) --  node (v1) {}  (0,0)--cycle;

\draw[thick] (root') -- (vd);
\draw[thick] (root') -- (v2);
\draw[thick] (root') -- (v1);

\node at (-6,6) {};
\end{tikzpicture}
 }
}
  \end{center}
  \caption{The self-similar frog model operator. Red sites contain particles that are initially active and blue sites contain initially dormant particles. $\A \pi$ is the law for the number of particle frozen at $\varnothing$ when the process fixates and $U$ is the number of vertices among $v_2,\hdots,v_d$ that are ever visited. Empty boxes on the right at $v_1,\hdots, v_d$ represent sites whose particles were activated.}
  \label{fig:Aasystem}
\end{figure}
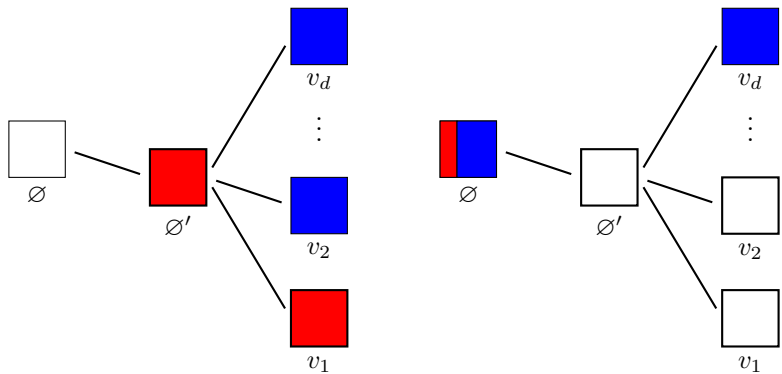

Given a probability measure $\pi$ on the nonnegative integers, we define  $\A \pi$ to be the \emph{self-similar frog model operator}. It is obtained from the following \emph{auxiliary process}. 

Consider a star graph with root $\varnothing$, central vertex $\varnothing '$, and leaves $v_1, \hdots, v_d$ (see Figure \ref{fig:Aasystem}). There is a $\Poi(1)$ number of active particles at $\varnothing'$ and a $\pi$-distributed number of active particles at $v_1$. Independent $\pi$-distributed numbers of dormant particles are placed at $v_2,\hdots, v_d$. 

The active particles started at $\varnothing'$ move to $\varnothing$ independently with probability $p^*$ and otherwise each moves to an independently and uniformly sampled vertex from $v_1,\hdots, v_d$. Active particles at $v_i$ move to $\varnothing '$ with probability $1$, and then to either $\varnothing$ with probability $\hat p$ or otherwise to a uniformly sampled vertex among $\{v_1,\hdots, v_d\} \setminus \{v_i\}$. Whenever active particles encounter dormant particles, the dormant particles become active. When a particle moves to a leaf, it remains frozen there for all subsequent time steps. 

Take $\A \pi$ to be the law for the total number of particles frozen at $\varnothing$ when the process fixates. Also, let $\mathcal U \pi$ be law for the total number of $v_2, \hdots, v_d$ that are ever visited by an active particle. We define $U= U(d,p,\lambda)$ to be a random variable with distribution $\mathcal U \Poi(\lambda)$. 

\subsection{Properties of $\A$}

The following facts state that the law of $V_{\SFM(d,p)}$ is a fixed point of $\A$, that $\A$ is monotone, and that $\A \Poi(\lambda)$ has a particularly nice representation. We also state \cite[Theorem 3.1 (b)]{misra2003stochastic} for comparing a Poisson random variable to one with a random parameter. We omit the proofs because they arise almost immediately from the construction and analogues have been observed in \cite{hoffman2016transience, johnson2016critical, junge2022stochastic}. We will abuse notation and write $\mathcal A V_{\SFM(d,p)}$ and $\mathcal A \Poi(\lambda)$ to represent the operator $\mathcal A$ applied to the associated probability measure.

\begin{fact}\thlabel{lem:fix}
$\A V_{\SFM(d,p)}\overset{d} = V_{\SFM(d,p)}$.
\end{fact}

\begin{fact} \thlabel{lem:mono}
If $\pi \preceq \pi'$, then $\A \pi \preceq \A \pi'$.
\end{fact}





\begin{fact}\thlabel{lem:poi}
 $\mathcal A \Poi(\lambda) \overset{d} = \Poi( p^* + \hat p(1+U) \lambda).$
\end{fact}


\begin{fact} \thlabel{lem:compare}
Suppose that $Y \sim \Poi(\Theta)$ and $Z \sim \Poi(\lambda)$ with $\Theta$ a nonnegative random variable and $\lambda \geq 0$. By \cite[Theorem 3.1 (b)]{misra2003stochastic}, the following are equivalent:
$$Y \succeq Z \iff \P(Y =0 )\leq \P(Z=0) \iff \E[e^{-\Theta}] \leq e^{-\lambda}.$$
\end{fact}

We apply these facts to give a sufficient condition for recurrence.

\begin{proposition} \thlabel{prop:suff}
   Suppose that $V_{\SFM(d,p)} \succeq \Poi(\lambda_0)$. Let $U' = U'(d,p,\lambda)$ be a family of random variables indexed by $\lambda$ with $U'(d,p,\lambda) \preceq U(d,p,\lambda)$ for all $\lambda \leq \lambda_0$. 
   
   If there exists $\epsilon >0$ such that for all $\lambda \geq \lambda_0$
    \begin{align}
         \E[e^{- p^* - \hat p (1+U')\lambda}] \leq e^{-\lambda - \epsilon},\label{eq:suff}
    \end{align}
    then $\SFM(d,p)$, $\NBFM(d,p)$, and $\FM(d,p)$ are recurrent.
\end{proposition}

\begin{proof}
Let $X=p^* + \hat p(1+U)\lambda$ and $X' = p^* + \hat p(1+U') \lambda $. If \eqref{eq:suff} holds, then \thref{lem:compare} implies that $\Poi( X' ) \succeq \Poi(\lambda + \epsilon)$. \thref{lem:poi} and our assumption that $U' \preceq U$ then imply that for all $\lambda \geq \lambda_0$ $$ \A \Poi(\lambda)  \overset{d}= \Poi(X) \succeq \Poi(X') \succeq \Poi(\lambda+ \epsilon).$$

Starting with $V_{\SFM(d,p)} \succeq \Poi(\lambda_0)$ and iteratively applying \thref{lem:fix} and \thref{lem:mono} gives 
$$V_{\SFM(d,p)} = \mathcal A^{(n)} V_{\SFM(d,p)} \succeq  \A^{(n)} \Poi(\lambda_0) \succeq \Poi(\lambda_0 + \epsilon n)$$ for all $n \geq 0$. It follows that $P(V_{\SFM(d,p)}=\infty)=1$. As $\SFM(d,p)$ is recurrent, \thref{lem:so} ensures that so are $\NBFM(d,p)$ and $\FM(d,p)$.
\end{proof}




    




\begin{remark} \thlabel{rem:balance}
The random variables $U', \tilde U$, and $U''$ from Sections~\ref{sec:sup} and \ref{sec:lim} are carefully balanced to satisfy \eqref{eq:suff}. Expanding \eqref{eq:suff}, we would like to show that
$$\E[e^{-p^* - \hat p (1+U') \lambda}] = \sum_{u=0}^{d-1} e^{- p^* - \hat p (1+u) \lambda} \P(U' = u) \leq e^{-\lambda - \epsilon}.$$
For $1+u\geq 1 / \hat p$, the summands are much smaller than $e^{-\lambda}$. However, when $1+u < 1 / \hat p$, we need good bounds on the probability coefficients $\P(U'=u)$ to make up for the $e^{-\hat p(1+u) \lambda}$ terms being too large on their own. The balancing act is modifying $U$ to obtain a smaller random variable with a tractable distribution that does not sacrifice too much precision.

\end{remark}

\section{Proof of \thref{thm:sup}} \label{sec:sup}

Fix $p=5/17$ so that $p^* = (5d-5)/(12d-5) \text{ and } \hat p =  5/{12}.$
Note that $p^*$ is easily seen to be increasing in $d$. By taking $d=3$ and $d=\infty$ we have 
\begin{align}
 10/31 \leq p^* \leq 5/12 \qquad \text{ for all $d \geq 3$} \label{eq:p*b}.
\end{align}

By \thref{prop:suff}, it suffices to find a random variable $U' \preceq U(d,p,\lambda)$ and $\epsilon >0$ so that \eqref{eq:suff} holds for all $d \geq 3$. We define $U'$ to be the number of activated leaves in the following \emph{modified auxiliary process}. First, we reduce to a star graph with only three leaves $v_1', v_2', v_3'$. Second, the $\Poi(1)$ active particles at $\varnothing'$ move to $\varnothing$ with probability $10/31\leq  p^*$, away from $\varnothing$ to a uniformly sampled leaf with probability $7/12\leq 1-p^*$, and otherwise are immediately killed. Besides these changes, the process evolves in the same manner as the auxiliary process and runs until fixation. Then $U' \in \{0,1,2\}$ is how many of $\{v_2', v_3'\}$ are eventually visited by an active particle.

\begin{lemma} \thlabel{lem:coupon}
    $U' \preceq U(d,5/17, \lambda)$ for all $d\geq 3$.
\end{lemma}
\begin{proof}
We first consider an intermediate process in which the active particles started at $\varnothing'$ independently move away from $\varnothing'$ to a uniformly sampled vertex from $v_1,...,v_d$ with probability $7/12$ rather than $1-p^*$. Let $U_d' \in \{0,1,\hdots, d-1\}$ be the number of activated leaves with this modification. 
Since $7/12<1-p^*$, less particles are being sent to the leaves and thus $U_d' \preceq U(d,5/17, \lambda)$. Now the probabilities that active particles move towards the leaves are the same in both processes defining $U_d'$ and $U'$. From here, it is straightforward to couple the two random variables so that $U' \preceq U_d'$ for all $d \geq 3$. It basically amounts to showing that a coupon collecting process with $d$ versus $2$ coupons has (stochastically) more unique coupons discovered after sampling the same number of coupons in each process.  Thus, $U' \preceq U(d,5/17, \lambda)$.
\end{proof}

\begin{proof}[Proof of \thref{thm:sup}]
Fix $p = 5/17$. We will show that \eqref{eq:suff} holds for all $\lambda \geq 0$. By \thref{lem:coupon}, and taking $\lambda_0=0$ in \thref{prop:suff}, we then have $\FM(d,5/17)$ is recurrent for all $d\geq 3$. 

Towards \eqref{eq:suff}, we use the bound $10/31 \leq p^*$ from \eqref{eq:p*b} and that $\hat p =5/12$ to write
\begin{align}
    \E[e^{- p^* - \hat p (1+ U') \lambda}] &=  \sum_{u=0}^2 e^{-p^*} e^{- \hat p  (1+u) \lambda} \P(U'=u)	\\
    &=  e^{-\lambda} \sum_{u=0}^2 e^{-p^*} e^{(1-\hat p (1 + u)) \lambda } \P(U'=u) \\
    &\leq e^{-\lambda} \sum_{u=0}^2 e^{-\f{10}{31}} e^{(1 - \f 5{12}  (1+u) ) \lambda} \P(U'=u). \label{eq:f1}
\end{align}
Set
\begin{align}
f(\lambda) \coloneqq \sum_{u=0}^2 e^{-\f{10}{31}} e^{(1 - \f 5{12}  (1+u) ) \lambda} \P(U'=u).\label{eq:f}
\end{align}
Deducing \eqref{eq:suff} comes down to proving that 
\begin{equation}
f(\lambda) \leq e^{-\epsilon} \text{ for all $\lambda \geq 0$ and some $\epsilon >0$.} \label{eq:fe} 
\end{equation} 

Let $U'$ be the number of leaves from $\{v_1',v_2',v_3'\}$ that are visited by an active particle in the modified auxiliary process with $\pi=\Poi(\lambda)$ for $\lambda\geq 0$. Poisson thinning  allows us to explicitly compute the distribution of $U'$.  Each of the $\Poi(1)$ active particles initially at $\varnothing'$ visits one of the vertices $\{v_2',v_3'\}$ with probability $\frac{7}{12}\cdot\frac{2}{3}$, and each of the $\Poi(\lambda)$ active particles initially at $v_1'$ move to $\varnothing'$ with probability $1$, and then visits one of the vertices $\{v_2',v_3'\}$ with probability $7/12$. Therefore, the probability that none of the vertices $\{v_2',v_3'\}$ is activated is the probability that a $\Poi(\frac{7}{12}\frac{2}{3}+\frac{7}{12}\lambda)$ distributed random variable is equal to $0$, giving rise to $\P(U'=0)$. One can apply the same argument to compute $\P(U'=1)$. The only difference is that if one of the vertices $\{v_2',v_3'\}$ is visited, particles initially there will be activated and can possibly visit the remaining unvisited vertex. 

As a result of these calculations, we obtain
\begin{align}
    \P(U'=0) &=  e^{ - \f 7{12} \f 2 3 -  \f 7{12} \lambda }  = e^{- \f 7{18} - \f 7{12} \lambda}\\
    \P(U'=1) &= 2 ( 1- e^{ - \f 7{12} \f 13 - \f 12 \f 7 {12} \lambda } )e^{ - \f 7{12} \f 13 - \f 12 \f 7 {12}  \lambda } e^{ - \f 12 \f 7 {12}  \lambda }  = 2 (1- e^{-\f 7{36}- \f 7{24} \lambda })e^{-\f 7{36} - \f 7{12} \lambda} \\
    \P(U'=2) &= 1  - \P(U'=0) - \P(U'=1).
\end{align}
Using these terms in \eqref{eq:f} and letting
\begin{align*}
g(x)
\coloneqq\frac{2 x^{27}}{e^{397/558}}-\frac{x^{20}}{e^{397/558}}-\frac{2 x^{20}}{e^{577/1116}}-\frac{2 x^{17}}{e^{397/558}}+\frac{2 x^{10}}{e^{577/1116}}+\frac{x^6}{e^{10/31}}+\frac{1}{e^{397/558}},
\end{align*}
one can verify that $g(e^{-\lambda /24}) = f(\lambda)$. So, \eqref{eq:fe} is equivalent to the statement
\begin{align}
 \label{eq:g} g(x) \leq e^{-\epsilon} \text{ for some $\epsilon >0$ and all $x \in [0,1]$}
\end{align}


To prove \eqref{eq:g} we first compute the derivative: 

\[g'(x) = \frac{54 x^{26}}{e^{397/558}}-\frac{20 x^{19}}{e^{397/558}}-\frac{40 x^{19}}{e^{577/1116}}-\frac{34 x^{16}}{e^{397/558}}+\frac{20 x^9}{e^{577/1116}}+\frac{6 x^5}{e^{10/31}}.\]
The algorithm $\texttt{CountRoots}_{[0,1]}[g']=6$ from Wolfram Mathematica applies Sturm's theorem to rigorously find that $g'$ has exactly $6$ roots in $[0,1]$. 
 Since $x=0$ is a root of order 5, $g'$ has exactly one root in $(0,1]$. Call it $r_0$. Elementary calculus shows that $g(r_0)$ is the global maximum on $[0,1]$. As $g$ is an explicit polynomial, mathematical software can rigorously estimate both $r_0$ and $g(r_0)$ to arbitrary precision. Doing so gives $g(r_0) \leq .9963 < e^{-.003}$.

Thus, we may take $\epsilon = .003$ in \eqref{eq:g}. This gives \eqref{eq:fe} for the same $\epsilon$ which implies \eqref{eq:suff}. As discussed at the onset of the proof, we have satisfied the hypotheses of \thref{prop:suff}. So, $\FM(d,5/17)$ is recurrent for all $d \geq 3$.
\end{proof}

Assuming the reader has familiarity with the proof of \thref{thm:sup}, we now prove the generalization in \thref{cor:4}. 

\begin{proof}[Proof of \thref{cor:4}]
    Fix $p=27/100$ so that $p^*=(27d-27)/(73d-27)$ and $\hat{p}=27/73$. Suppose that $d \geq 4$. Since $p^*$ is increasing in $d$, by taking $d=4$ and $d=\infty$, we have for all $d\geq 4$,
    \[
    \frac{81}{265}\leq p^*\leq \frac{27}{73}.
    \]
    As in the proof of \thref{thm:sup}, we will show that \eqref{eq:suff} holds for all $\lambda \geq 0$. 

We define $\tilde U:= \tilde U (d,p,\lambda)$ to be the number of activated leaves in the following \emph{modified auxiliary process}. First, we reduce to a star graph with only four leaves $\tilde v_1, \tilde v_2, \tilde v_3, \tilde v_4$. Second, the $\Poi(1)$ active particles at $\varnothing'$ move to $\varnothing$ with probability $\tp = 81/265\leq p^* $,
away from $\varnothing$ to a uniformly sampled leaf with probability $\rho = 46/73\leq 1-p^*$,
and otherwise are immediately killed. Besides these changes, the process evolves in the same manner as the auxiliary process and runs until fixation. The random variable $\tilde U \in \{0,1,2,3\}$ is how many of $\{\tilde v_2, \tilde v_3, \tilde v_4\}$ are ever visited by an active particle. Following a similar argument as \thref{lem:coupon}, we have $\tU \preceq U(d,27/100, \lambda)$ for all $d\geq 4$. 

Using Poisson thinning, we can compute the distribution of $\tilde U$:
\begin{align}
    \P(\tU =0) &= e^{-\rho \f 34 - \rho \lambda} \\
    \P(\tU =1 ) &= 3 (1- e^{- \rho \f 14 - \rho \f 13 \lambda } ) e^{-\rho \f 2 4 - \rho \f 23 \lambda } e^{- \rho \f 2 3 \lambda } \\
    \P(\tU = 2) &= 3 (1- e^{- \rho \f 14 - \rho \f 13 \lambda } )^2 e^{-\rho \f 1 4 - \rho \f 13 \lambda } e^{- \rho \f 1 3  2\lambda }  \\
            & \hspace{ 1 cm} + 3 (1- e^{- \rho \f 14 - \rho \f 13 \lambda } ) e^{-\rho \f 2 4 - \rho \f 23 \lambda }  \\
                    & \hspace{2 cm} \times 2 (1- e^{- \rho \f 13 \lambda } ) e^{- \rho \f 1 3 \lambda }\\
    \P(\tU = 3) &= 1 - \P(\tU \leq 2)
\end{align}
In words: $\{\tU = 0\}$ has no frogs from $\tilde v_1$ move to $\tilde {\mathcal V} \coloneqq \{\tilde v_2, \tilde v_3, \tilde v_4$\}; $\{\tU = 1\}$ has one vertex from $\tilde {\mathcal V}$ become activated (3 choices) and the other two fail to activate; and $\{\tU = 2\}$ has either two vertices from $\tilde {\mathcal V}$ initially activate (3 choices) and the third fail to activate, or one vertex from $\tilde {\mathcal V}$ initially activate (3 choices) and that activates exactly one more (2 choices), which then fails to activate the remaining vertex. 

As in the proof of \thref{thm:sup}, we can write
$$\E[e^{-p^* - \hat p (1+\tU) \lambda }] \leq e^{-\lambda} \tilde f(\lambda)$$ with
\begin{align}
\tilde f(\lambda) = \sum_{u=0}^3 e^{-\tp} e^{(1- \rho (1 +u) ) \lambda}\P(\tU = u).\label{eq:tf}
\end{align}

It suffices to prove that $\tilde f(\lambda) \leq e^{-\epsilon}$ for some $\epsilon >0$ and all $\lambda \geq 0$. After the change of variables $\lambda \to -219 \log x$, this is equivalent to proving that $\tilde g(x) \coloneqq \tilde f(- 219 \log x) \leq e^{-\epsilon}$ for some $\epsilon >0$ and all $x \in [0,1]$. The choice $219$ was made from inspecting the expansion of $\tilde f$ (computed with Mathematica) to find the least common denominator of the fractional exponents involving $x$. We check that \texttt{CountRoots}$_{[0,1]}[\tilde g'] = 70$. Since $\tilde g'$ has a root of multiplicity $69$ at $x=0$, elementary calculus can be used to show that $\tilde g$ has a global maximum in $[0,1]$ at $\tilde x_0 \approx 0.992241$ of $\tilde g(x_0) \approx 0.998772 < e^{-.0011}$. These approximations are within $10^{-7}$ of the true values, so we may take $\epsilon =0.001$ and complete the argument as in the proof of \thref{thm:sup}.
\end{proof}

\begin{remark} \thlabel{rem:m}
    We describe how to generalize \thref{cor:4} to obtain a bound on $\sup_{d \geq m} p_d$. In principle, as $m$ increases, this should give bounds closer and closer to $1/6$ in agreement with \thref{thm:limsup}. We did not try to go beyond $m=4$, but this will likely become computationally infeasible at $m\approx 10$.  First we replace the $d$ leaves with $m$ leaves and  construct $\tilde U_m \in \{0,1,\hdots, m-1\}$ that is stochastically smaller than $U(d,p,\lambda)$. This is accomplished by using the drift $\tilde p_m = p^*(m,p) \leq p^*(d,p)$ towards $\varnothing$ in the auxiliary process and the drift $\rho_m = 1-\hat{p} \leq 1-p^*(d,p)$ away from $\varnothing$ to a uniformly sampled child vertex from $v_1,...,v_m$. We then need to compute the distribution of $\tilde U_m$ exactly. This is theoretically possible for any $m$, but becomes more and more complex as $m$ grows. One then constructs a function $\tilde f_m(\lambda) = \sum_{u=0}^{m-1} e^{- \tilde p_m} e^{(1- \rho(1+u))\lambda) }\P(\tilde U_m = u)$ as at \eqref{eq:tf}. One can plot $\tilde f_m$ using mathematical software to approximate small value of $p$ for which $\tilde f_m(\lambda) <1$ for all $\lambda \geq 0$, then use our approach that employs $\texttt{CountRoots}$ to show that the transformation $\tilde g_m(x) < 1$. 
\end{remark}

\section{Non-backtracking branching random walk} \label{sec:NBBRW}
In this section, we construct and deduce some properties of various spatially homogeneous multitype branching random walks that relate back to the frog model.

\subsection{Construction}
The process starts with a configuration of particles on $\mathbb Z_{+}$ at time 0. Each particle comes with a type $i\in \{1,2,...,k\}$. At discrete time steps, each particle independently gives birth to a random number of particles according to an offspring distribution that only depends on its type. The parent particle dies immediately after. Each newborn particle independently moves according to some displacement distribution that only depends on the particle's type. Particles that reach 0 are stopped there instantaneously and stay there forever without producing any offspring. 

For $i,j=1,...,k$, let $r_{ij}$ be the expected number of offspring of Type-$j$ produced by one Type-$i$ particle and $R=(r_{ij})_{i,j=1,...,k}$ be the mean matrix of the offspring distributions. 
For a particle of Type-$i$ at site $x\in\mathbb Z_{\geq 0}$, we let 
$$p_{x,y}^i = \P(\text{Type $i$ born at $x$ moves to $y$}).$$

We now give a formal definition of the {\it non-backtracking $p$-biased branching random walk on the nonnegative integers} which we denote by $\NBBRW(p)$. Suppose a given particle is at $x \geq 1$. Particles die immediately after producing offspring in the following manner: 
    \begin{description}
        \item[Type-1] Correspond to active particles that have yet to start moving away from the root. Each such particle produces either one Type-1 offspring with probability $\hat p$, or one Type-2 offspring plus a $\Poi(1)$-distributed number of Type-3 offspring with probability $1-\hat p$.
        \item[Type-2] Correspond to active particles that have began to move away from the root. Each such particle produces one Type-2 offspring and a $\Poi(1)$-distributed number of Type-3 offspring with probability $1$.
         \item[Type-3] Auxiliary Type-1 particle, have the same offspring distribution as Type-1 particles, but different displacement distribution. 
    \end{description}
After producing offspring, each newly generated particle, independently of everything else, displaces from $x$ according to the following transition probabilities:
\[
p_{x,x-1}^{1}=1,\quad  p_{x,x+1}^{2}=1, \quad p_{x,x+1}^{3}=1,\quad\text{for }x>0.
\]
In words, Type-1 particles always move one step left, and Type-2 and Type-3 particles move one step right. We stop any particles that reach $0$. 

Some quick remarks:
 \begin{itemize} 
   \item Type-2 particles correspond to non-backtracking active frogs that have turned away from the root and will continue moving away for all steps. Type-1 and Type-3 particles correspond to non-backtracking active frogs that may still jump towards the root. We need two different particle types so that the displacements are independent of the manner in which particles are born.
    \item There is no dependence on $d$ in the definition of $\NBBRW(p)$. Since $p^* \to \hat p$ when $d \to \infty$, one can view $\NBBRW(p)$ as the intuitive limiting version of $\NBFM(d,p)$.

    \item The mean displacement matrix $R = (r_{ij})$ is
\begin{equation*}
R=\begin{bmatrix}
\hat p & 1-\hat p & 1-\hat p\\
0 & 1 & 1\\
\hat p &1-\hat p & 1-\hat p
\end{bmatrix}.
\end{equation*}
  
\end{itemize} 


Let $V_{\NBBRW(p)}$ be the total number of particles that are killed at the origin. We say that $\NBBRW(p)$ is \emph{recurrent} if $\P(V_{\NBBRW(p)} = \infty)=1$ and otherwise \emph{transient}. We adapt ideas from \cite{comets1998one, machado2001recurrence} to find the criteria of recurrence and transience. 

Let us first introduce some additional notation from \cite{machado2001recurrence}. For $i=1,2,3$, we denote by $N_i(t)$ the number of Type-$i$ particles at time $t$ and $\{X^{i}_k(t): k=1,...,N_i(t)\}$ the set of positions of Type-$i$ particles at time $t$. Then the configuration at time $t$ is the multiset 
\[
\omega(t)=\{X_1^1(t),...,X_{N_1(t)}^1(t),...,X^3_1(t),...,X^3_{N_3(t)}\}.
\] 
The configuration of Type-$i$ particles at time $t$ is the multiset
\[
\omega_i(t)=\{X^i_1(t),...,X^i_{N_i(t)}(t)\}.
\]
We denote $p_i=p_{x,x-1}^{i}$ and $q_{i}=p_{x,x+1}^{i}$. Then we have $p_1=q_2=q_3=1$ and $p_2=p_3=q_1=0$.

Let $\mathcal M$ be the collection of initial configurations that consist of a finite number of particles distributed on $\mathbb Z_{\geq 0}$. Since for any $i,j=1,2,3$, Type-$i$ particles can be generated by a Type-$j$ particle in finite steps with positive probability, we have either $\E[V_{\NBBRW(p)}\mid \omega(0)=\omega]<\infty$ for all $\omega\in\mathcal M$ or $\E[V_{\NBBRW(p)} \mid \omega(0)=\omega]=\infty$ for all $\omega\in\mathcal M$. We will omit the initial configuration when we only care about the finiteness of $\E[V_{\NBBRW(p)}]$ rather than its precise value. 

\subsection{Transience and recurrence criteria}

The following lemma is a combination of \cite[Theorem 4 and Theorem 7]{machado2001recurrence}. It gives both necessary and sufficient conditions for $\NBBRW(p)$ to have a finite expected number of particles hitting the origin.  The proof is a non trivial application of \cite[Theorem 4 and Theorem 7]{machado2001recurrence}, because (a) $\NBBRW(p)$ does not satisfy all of the hypotheses used in \cite{machado2001recurrence}, and (b) the definitions of recurrence and transience in \cite{machado2001recurrence} are different from our definitions. However, the proof ideas can be adapted to our case. We also note that results similar to \thref{lem:VnbBRW} are present under other settings. A more general from of \eqref{eq:Thm7} appeared first in \cite{karpelevich1994boundedness} as a classification of one-dimensional branching random walk, then in \cite{menshikov1997branching} as a qualitative characterization of recurrence and transience for branching Markov chains, and also in \cite{comets1998one} under the setting of one-dimensional branching random walk in a random environment. The proof ideas are in the same vein.

\begin{lemma}\thlabel{lem:VnbBRW}
Consider $\NBBRW(p)$ started from a finite number of particles. If there exist $\mu>0,\alpha_1,\alpha_2,\alpha_3>0$ such that for $i=1,2,3$
\begin{equation}\label{eq:Thm7}
\sum_{j=1}^{3}r_{ij}\alpha_j\left(p_j\frac{1}{\mu}+q_j\mu\right)\leq \alpha_i,
\end{equation}
then $\E[V_{\NBBRW(p)}]<\infty$. On the other hand, if $\E[V_{\NBBRW(p)}]<\infty$, then there exist $\mu>0,\alpha_1,\alpha_2,\alpha_3>0$ such that \eqref{eq:Thm7} holds with equality for $i=1,2,3$.
\end{lemma}

\begin{proof}
If \eqref{eq:Thm7} holds, define
\[
Q(t)=\sum_{i=1}^{3}\sum_{j=1}^{N_i(t)}\alpha_i\mu^{X_j^i(t)}.
\]
The process $\{Q(t)\}_{t=0}^{\infty}$ is a non-negative supermartingale. Indeed, let $\mathcal F_t$ be the $\sigma$-field generated by $\NBBRW(p)$ up to time $t$. By the branching property and \eqref{eq:Thm7}, we have
\begin{align*}
&\E[Q(t+1) \mid \mathcal F_t]\\
&=\sum_{i=1}^{3}\sum_{j=1}^{N_i(t)}\left(\sum_{k=1}^3 r_{ik}\alpha_k\left(p_k \mu^{X_j^i(t)-1}+q_k\mu^{X_j^i(t)+1}\right)1_{\{X_j^i(t)>0\}}+\alpha_i \mu^{X_j^i(t)}1_{\{X_j^i(t)=0\}}\right)\\
&= \sum_{i=1}^{3}\sum_{j=1}^{N_i(t)}\left(\mu^{X_j^i(t)}\sum_{k=1}^3 r_{ik}\alpha_k\left(p_k \frac{1}{\mu}+q_k\mu\right)1_{\{X_j^i(t)>0\}}+\alpha_i \mu^{X_j^i(t)}1_{\{X_j^i(t)=0\}}\right)\\
&\leq Q(t).
\end{align*}
By the supermartingale convergence theorem, there exists a random variable $Q_{\infty}$ such that $Q(t)\rightarrow Q_{\infty}$ almost surely as $t\rightarrow\infty$ and $\E[Q_{\infty}]\leq \E[Q(0)]$. We further note that $Q_{\infty}\geq \alpha_1 V_{\NBBRW(p)}$ since only Type-1 particles can hit the origin. As a result,
\[
\E[V_{\NBBRW(p)}]\leq \frac{\E[Q_{\infty}]}{\alpha_1}\leq \frac{\E[Q(0)]}{\alpha_1}<\infty.
\]

On the other hand, if $\E[V_{\NBBRW(p)}]<\infty$, we define for $x\in \mathbb Z_{\geq 0}$
\begin{equation}\label{eq:fi(x)}
f_i(x)= \E[V_{\NBBRW(p)} \mid \omega(0)=\omega_i(0)=\{x\}]
\end{equation}
to be the expected total number of visits to $0$ conditional on the initial configuration starting with a single particle of type $i$ at $x$.
Note that $f_i(1)>0$ for $i=1,2,3$. By the first step analysis, we have for $i=1,2,3$ and $x\in \mathbb Z_{\geq 1}$,
\begin{equation}\label{eq:Thm4}
f_i(x)=\sum_{j=1}^{3} r_{ij}\left(p_j f_{j}(x-1)+q_j f_j(x+1) \right).
\end{equation}

For $x\geq 1$, consider $\NBBRW(p)$ started from a single particle at $x+1$ with Type-$i$. We can construct a modified process in which particles that reach the site $1$ are stopped. Let $V_1$ be the number of particles that reach $1$ and are stopped. For each lineage, only Type-$1$ particles can reach $1$ for the first time. Note that $V_1$ has the same distribution as $V_{\NBBRW(p)}$ under the process started from one Type-$i$ particle at $x$. Therefore, $V_1<\infty$ almost surely. Furthermore, because each Type-$1$ particle that reaches $1$ behaves afterwards like another $\NBBRW(p)$ started from a single particle at $1$ with Type-$1$, the number of particles stopped at $0$ conditioned on $V_1$ is the same as the distribution of the sum of $V_1$ independent random variables, each with the same distribution as $V_{\NBBRW(p)}$ under the process started from one Type-$1$ particle at $1$. We therefore have for $i=1,2,3$
\[
f_i(x+1)=f_i(x)f_1(1).
\]
By induction, we get for $i=1,2,3$ and $x\in\mathbb Z_+$
\begin{equation}\label{eq:induction}
f_i(x)=f_i(1)f_1(1)^{x-1}.
\end{equation}
Plugging \eqref{eq:induction} into \eqref{eq:Thm4}, by choosing $\mu=f_1(1)$ and $\alpha_i=f_i(1)$ for $i=1,2,3$, equation \eqref{eq:Thm7} holds with equality and the lemma follows.
\end{proof}

\begin{lemma} \thlabel{lem:nbBRW}
Suppose the initial configuration is finite and contains at least one Type-2 particle not at 0. Then  $\NBBRW(p)$ is transient if and only if $p \leq 1/6.$
\end{lemma}
\begin{proof}
\thref{lem:VnbBRW} gives a criteria for proving that $\NBBRW(p)$ is transient. Namely, it is sufficient to prove that, when $p \leq 1/6$, there exist $\theta\in\mathbb R$ and $\alpha_1,\alpha_2,\alpha_3>0$ such that equation \eqref{eq:Thm7} holds with equality for $i=1,2,3$.
Given $\theta \in \mathbb R$, define the weight matrix
$$
\Phi(\theta) = \begin{bmatrix}
    {\hat p} e^{\theta} + (1-{\hat p}) e^{-\theta} & (1-{\hat p}) e^{-\theta} \\
    e^{-\theta} & e^{-\theta}
\end{bmatrix}.
$$ 
Equation \eqref{eq:Thm7} would hold with equality if we can show that 1 is the eigenvalue of $\Phi(\theta)$ and there exists an eigenvector associated to $1$ with all of its elements positive. 

The eigenvalues of $\Phi(\theta)$ are
\begin{align}
    \beta_{\pm} &= \f 1 2 e^{-\theta} \left(2 - {\hat p} + e^{2 \theta} {\hat p} \pm 
   \sqrt{-4 e^{2 \theta} {\hat p} + (-2 + {\hat p} - e^{2 \theta} {\hat p})^2}\right),
\end{align}
with associated eigenvectors
\begin{align}
    v_\pm = \frac{1}{2} \left(-\hat p+e^{2 \theta} 
    \hat p \pm\sqrt{-4 e^{2 \theta} {\hat p} + (-2 + {\hat p} - e^{2 \theta} {\hat p})^2},1\right)
\end{align}
Solving $\beta_{\pm} =1$ for $\theta$ gives the same solutions for both eigenvalues. In particular,
$$\theta_\pm = \log\left[(1 + {\hat p} \pm \sqrt{1 - 6 {\hat p} + 5 {\hat p}^2})/(2 {\hat p})\right].$$
For $\hat p \in (0,1/5]$, the quadratic $1- 6{\hat p} + 5{\hat p}^2$ is non-negative and there is a solution. This implies that for $\hat p\in (0, 1/5]$, equivalently $p\in (0,1/6]$, there exists $\theta_0$ with $1$ the eigenvalue of $\Phi(\theta_0)$. Furthermore, it can be easily shown that $v_{+}$ is an eigenvector of $\beta_+=1$ and every element of $v_{+}$ is positive. Equation \eqref{eq:Thm7} follows with equality and thus $\NBBRW(p)$ is transient for any $p \leq 1/6$. 


On the other hand, if $1/6<p<1/2$, then $\hat{p}\in (1/5,1)$. Since the quadratic $1- 6{\hat p} + 5{\hat p}^2$ is negative for $\hat{p}\in (1/5,1/2)$, there are no solutions to \eqref{eq:Thm7} in which equality holds. We know from \thref{lem:VnbBRW} that $\E[V_{\NBBRW(p)}]=\infty$ for any finite initial configuration with not all particles at 0. It follows from the proofs of \cite[Theorem 9]{machado2001recurrence} and \cite[Theorem 4.3]{comets1998one} that since $\NBBRW(p)$ is homogeneous in the sense that offspring distributions and transition probabilities do not depend on the location of particles, $\P(V_{\NBBRW(p)}\geq 1)=1$ if the process starts from a finite number of particles not all located at 0. Moreover, if the initial configuration contains at least one Type-2 particle not at 0, then $\NBBRW(p)$ has Type-2 particles survive forever. Together with the Markovian property of $\NBBRW(p)$, we conclude that the origin is visited infinitely often almost surely. Therefore, $\P(V_{\NBBRW(p)}=\infty)=1$ and $\NBBRW(p)$ is recurrent.
%
\end{proof}

Lastly, it is necessary for our arguments to deduce transience of a reflected version of $\NBBRW(p)$. Let \emph{reflected non-backtracking branching random walk} $\RNBBRW(p)$ be the variant in which any particle that moves to $0$, instead of being stopped, converts to a Type-2 particle that continues producing offspring. In $\RNBBRW(p)$, particles reflect at the origin. 

\begin{lemma}\thlabel{lem:reflection}
    If $p \leq 1/6$, then $\RNBBRW(p)$ is transient. 
\end{lemma}

\begin{proof}
Let $V_{\RNBBRW(p)}$ be the number of times that particles hit the origin. Let $A_{i,x}$ denote the event $\omega(0) = \omega_i(0) = \{x\}$.  It is sufficient to prove that if $p \leq 1/6$, then for all $x\in\mathbb Z_{\geq 0}$ and $i=1,2,3$
\begin{equation}\label{eq:VrnbBRW}
\E[V_{\RNBBRW(p)} \mid A_{i,x}]<\infty.
\end{equation}

If $x=0$, then equation \eqref{eq:VrnbBRW} is obvious. Suppose $\RNBBRW(p)$ starts with one Type-$i$ particle at site $x\in \mathbb Z_{>0}$. If a particle hits the origin at time $t$, we can trace its past trajectory and count the total number of times this particle has hit the origin until time $t$. Call a visit to the origin an $n$th visit if the visiting particle has visited the origin exactly $n-1$ times in its past trajectory. Let $V_{\RNBBRW(p)}(n)$ be the total number of $n$th visits. We can decompose the expectation in  \eqref{eq:VrnbBRW} as
\begin{equation}\label{eq:VrnbBRWsum}
\E[V_{\RNBBRW(p)} \mid  A_{i,x}]=\sum_{n=1}^{\infty} \E[V_{\RNBBRW(p)}(n) \mid A_{i,x}].
\end{equation}

For each $n\geq 1$, we consider a modified process $\RNBBRW_n(p)$ in which all particles are killed immediately after an $n$th visit. Note that for each $n\geq 1$, the number of particles killed at the origin in $\RNBBRW_n(p)$ is indeed $V_{\RNBBRW(p)}(n)$. Furthermore, only Type-1 particles can visit the origin. 
In the original process $RNBBRW_n(p)$, the Type-1 particle is not killed after the visit. Instead, it will convert to a Type-2 particle and generate one Type-2 particle and $\Poi(1)$ Type-3 particle at site 1 in the next step. When $n=1$, the modified process $\RNBBRW_1(p)$ is identical to $\NBBRW(p)$. 

Recall the functions $\{f_i(x)\}_{\mathbb Z_{>0}}$ defined in \eqref{eq:fi(x)} under the setting of $\NBBRW(p)$ for $i=1,2,3$. We have
\[
\E[V_{\RNBBRW(p)}(1) \mid  A_{i,x}]=f_i(x),
\]
which is finite when $p \leq 1/6$. When $n=2$, we can couple $\RNBBRW_2(p)$ with $\RNBBRW_1(p)$ such that all particles which hit the origin twice are descendants of particles that are killed in $\RNBBRW_1(p)$. Each particle that should have been killed in $\RNBBRW_1(p)$ will give birth to on average one Type-2 particle and one Type-3 particle at 1 in the next step. All of these newly generated particles will initiate independent copies of $\RNBBRW_1(p)$ (i.e. $\NBBRW(p)$) from site 1. Thus we obtain
\[
\E[V_{\RNBBRW(p)}(2) \mid  A_{i,x}]=f_i(x)(f_2(1)+f_3(1)).
\]
By induction, we have for all $n\geq 1$,
\begin{equation}\label{eq:formulaV(n)}
\E[V_{\RNBBRW(p)}(n) \mid  A_{i,x}]=f_i(x)(f_2(1)+f_3(1))^{n-1}.
\end{equation}
Therefore, equation \eqref{eq:VrnbBRW} would follow from \eqref{eq:VrnbBRWsum} and \eqref{eq:formulaV(n)} once we prove that 
\begin{equation}\label{eq:f23(1)}
f_2(1)+f_3(1)<1
\end{equation}
when $p \leq 1/6$.

It remains to compute $f_2(1)$ and $f_3(1)$. Recall the proofs and notation in \thref{lem:VnbBRW}. When $p\leq 1/6$, $f_i(x)$ is finite for all $x\in\mathbb Z_{>0}$ and $i=1,2,3$. Since $\{f_i(x)\}_{\mathbb Z_{>0}}$  satisfies equations \eqref{eq:Thm4} and \eqref{eq:induction}, we have $f_1(1)=\mu<1$ and for all $n\geq 1$,
\begin{align}
f_1(n)&=f_3(n)=\mu^n\\
f_2(n)&=\sum_{k=n+1}^{\infty}f_3(k)=\sum_{k=n+1}^{\infty}\mu^k=\frac{\mu^{n+1}}{1-\mu}.
\end{align}
According to the computation in \thref{lem:nbBRW}, we observe that
\[
f_1(1)=\mu=e^{-\theta}=\frac{2 {\hat p}}{1 + {\hat p} \pm \sqrt{1 - 6 {\hat p} + 5 {\hat p}^2}}=\frac{1+\hat p \mp \sqrt{1-6\hat p+5\hat p^2}}{2(2-\hat p)}.
\]
When $p \in (0, 1/6]$, i.e. $\hat p\in (0,1/5]$, one can easily check that 
\[
0<\mu\leq \frac{1+\hat p + \sqrt{1-6\hat p+5\hat p^2}}{2(2-\hat p)}<1/2,
\]
which implies that
\[
f_2(1)+f_3(1)=\frac{\mu^2}{1-\mu}+\mu<1.
\]
Consequently, equation \eqref{eq:f23(1)} holds and the lemma follows.
\end{proof}


\subsection{Main ingredients in the proof \thref{thm:lim}} \label{sec:relate}

In \thref{lem:transient}, we show that transience of $\RNBBRW(p)$ when $p\leq 1/6$ implies transience of $\NBFM(d,p)$. In \thref{lem:lb}, we use recurrence of $\NBBRW(p)$ when $p > 1/6$ to show that $V_{\SFM(d,p)}$ dominates a Poisson random variable whose parameter diverges with $d$.

In regards to \thref{lem:transient}, intuitively $\NBBRW(p)$ ought to have more visits to $0$ than $\NBFM(d,p)$ to the root. The reasons are (a) particles in $\NBBRW(p)$ have a slightly stronger drift towards $0$ (because $p^*<\hat{p}$ for $p<1/2$), and (b) particles moving away from the root in $\NBBRW(p)$ always ``activate" an additional $\Poi(1)$-number of particles. As discussed in the introduction, it is generally not known how to couple models with different drifts. Fortunately, (b) is enough to overcome these complications. 

Overcoming the complications has the cost of a more involved coupling than might on the surface seem necessary.
For example, we need to work with reflected branching random walk $\RNBBRW(p)$. Otherwise, in the killed-at-0 version ($\NBBRW(p)$) the stronger drift might cause some particles to reach $0$ and be killed, which hurts total progeny. We also introduce a family of (reflected) branching processes whose particle displacements depend on $d$. These are nice intermediaries that couple more cleanly with $\RNBBRW(p)$ and $\NBFM(d,p)$.

\begin{lemma} \thlabel{lem:transient}
If $p\leq 1/6$, then $\NBFM(d,p)$ is transient. 
\end{lemma}

\begin{proof}
For this proof we will view $\NBBRW(p)$ as the process in which newly generated particles iteratively jump towards 0 with probability $\hat p$, but once they turn away, continue to jump away from $0$ for all subsequent steps. Each jump away from $0$ produces an independent $\Poi(1)$-distributed number of particles at the site jumped to. From this point of view, particles initially jump towards the root, eventually turn away, and then produce particles at each site they jump to thereafter.


Let $\RNBBRW(p)$, as introduced in \thref{lem:reflection}, be the \emph{reflected} modification. Any particle that visits 0, will on the next step jump to $1$ and produce an additional $\Poi(1)$-distributed number of particles there. Lastly, we define $\RNBBRW(d,p)$ to be the modification of $\RNBBRW(p)$ in which newly generated particles jump left on their first step with probability $p^*$ rather than $\hat p$. Subsequent steps are to the left with probability $\hat p$, as usual. 

These transition probabilities are chosen so that $\RNBBRW(d,p)$ stochastically dominates $\NBFM(d,p)$. Namely, any active frog in $\NBFM(d,p)$ can be coupled with a unique particle in $\RNBBRW(d,p)$ whose position is equal to the active frog's displacement from the root of $\mathbb T_d$. The coupling is intuitive and works because (i) we may view $\RNBBRW(d,p)$ as the variant of $\NBFM(d,p)$ in which every jump away from the root activates new particles, and (ii) the transition probabilities towards and away from the root and 0 are the same for all steps of a particle's life in both models. 
This coupling ensures that transience of $\RNBBRW(d,p)$ implies transience of $\NBFM(d,p)$. 

Since \thref{lem:reflection} gives transience of $\RNBBRW(p)$ whenever $p \leq 1/6$, it suffices to prove that transience of $\RNBBRW(p)$ implies transience of $\RNBBRW(d,p)$. To this end, we may couple the initial particle at $0$ in both $\RNBBRW(d,p)$ and $\RNBBRW(p)$ to introduce the same number of particles at each step away from the root. Since $p^* < \hat p$, any subsequently introduced particle in $\RNBBRW(d,p)$ can be coupled with a unique particle in $\RNBBRW(p)$, so that the particle in $\RNBBRW(p)$ moves at least at close to $0$ before turning away. From there, we may couple the number of particles the two particles generate on their path to $\infty$ to be the same at each jump. This ensures that each particle in $\RNBBRW(d,p)$ corresponds to a unique particle in $\RNBBRW(p)$ that starts at least as close to $0$ and moves at least as close to $0$ as its counterpart. Thus, there are stochastically fewer total visits to 0 in $\RNBBRW(d,p)$. 
\end{proof}

\begin{lemma} \thlabel{lem:lb}
If $p > 1/6$, then $V_{\SFM(d,p)} \succeq \Poi(\lambda_d)$ for some sequence $\lambda_d \to \infty$.
\end{lemma}
\begin{proof}
Note that $p^* \to \hat p$ as $d\to \infty$, and the probability of an active particle moving to an unvisited site within the first $t$ time steps converges to $1$ as $d \to \infty$. Thus, the probability that $\SFM(d,p)$ and $\NBBRW(p)$ couple to have the exact same behavior for the first $t$ time steps converges to $1$ as $d \to \infty$ for any fixed $t \geq 0$. By \thref{lem:nbBRW}, the probability of no root visits in $\NBBRW(d,p)$ goes to 0 as $t \to \infty$. Letting $V_{\SFM(d,p)}(t)$ be the number of root visits in $\SFM(d,p)$ up to time $t$, we then have $\P(V_{\SFM(d,p)}(t) = 0) \to 0$ as $d \to \infty$. Poisson thinning ensures that the number of root visits in $\SFM(d,p)$ has a Poisson distribution with random mean. By \thref{lem:compare}, $V_{\SFM(d,p)} \succeq \Poi(\lambda_d)$ for some sequence $\lambda_d \to \infty$. 
\end{proof}

\section{Proof of \thref{thm:lim}}\label{sec:lim}
We use the criteria from Section~\ref{sec:sup} to prove that the limiting critical drift for $\NBFM(d,p)$ is the same as that in $\NBBRW(p).$ The basic idea is to construct a random variable $U'' \preceq U$ from Section~\ref{sec:A}. We make it stochastically smaller by only allowing one leaf to be activated at a time. By taking $d$ large, we are able to get sufficiently strong bounds on the probability that $U''$ takes a small value.

Before defining $U''$, we describe an alternate way to sample $U$ via an \emph{exploration process}. Let $A_1$ be the set of leaves among $v_2,...,v_d$ that are visited by active particles started from $\varnothing'$ and $v_1$, and $C_1$ be the empty set. If $A_1$ is empty, then the process terminates. If not, then select a leaf $v$ from $A_1$ and allow the activated particles there to move until reaching $\varnothing$ or a leaf. Let $B_1$ be the set of leaves among $\{v_2,...,v_d\}-\{v\}$ that are visited for the first time by particles from $v$. Set $A_2 = A_1 \cup B_1 - \{v\}$ and $C_2 = C_1 \cup \{v\}$. Continue in this fashion to form $A_{n+1}$ by removing a leaf $v$ from $A_n$, adding $v$ to $C_{n+1}$, and adding the leaves visited for the first time by particles started from $v$ to $A_{n+1}$. Once $A_N$ is empty (which occurs after at most $d-1$ steps), we have $C_N$ is the set of leaves that are activated and so $U =|C_N|$. 

We define $U'' \preceq U$ by modifying the exploration process. When $B_n$ is non-empty, instead of adding all of its leaves to $A_n$ to form $A_{n+1}$, we choose a single leaf from $B_n$ and add it to $A_{n+1}$. We ignore the visit of any other leaves in $B_n$. In later rounds, all the leaves in $B_n$ except the one that was selected act as if the particles there are still dormant. Let $U''$ be the number of vertices activated in this modified process. 

Since we are potentially ignoring visits to dormant particles 
we have $U'' \preceq U.$ Moreover, for any $1\leq j\leq 4$ we have 
\begin{align}
    \P(U'' =j ) \leq e^{- (1- \hat p) \lambda (d-5)/(d-1)}.\label{eq:U''j}
\end{align}
This is because in order to have $U''=j$, the $j$th exploration must fail to visit any leaves with dormant particles. Since we are looking at the first $j\leq 5$ steps, there are always at least $d-5$ dormant leaves. Using Poisson thinning, the number of particles moving to leaves with dormant particles dominates a Poisson random variable with mean $(1- \hat p) \lambda (d-5)/(d-1)$. So the probability of failure at the $j$th step, is bounded by the probability that this Poisson distribution is 0, which is \eqref{eq:U''j}.

\begin{lemma} \thlabel{lem:U''}
    Let $U''$ be as defined above. Given $p > 1/6$, there exists $\lambda_0 \geq 0$ such that for all sufficiently large $d$ and $\lambda \geq \lambda_0$ 
        \begin{align}
        \E[ e^{-p^* - \hat p (1+ U'') \lambda} ] \leq e^{- \lambda - \f 18}.\label{eq:U''}
        \end{align}
\end{lemma}

\begin{proof}
    We first note that if $p>1/6$, then $p^*\geq 1/7$ for all $d\geq 3$. We expand the expectation at \eqref{eq:U''} to obtain
    \begin{align}
        \E[ e^{-p^* - \hat p (1+ U'') \lambda }]  & = e^{-p^*  } \sum_{u=0}^{d-1} e^{ - \hat p (1+u) \lambda } \P(U''=u)\\
        &\leq e^{-\f 1 7} \sum_{u=0}^{d-1} e^{ - \hat p (1+u) \lambda } \P(U''=u).
    \end{align}
Setting the $e^{-\f 17}$ factor aside for a moment, we decompose the sum into three parts:
    \begin{align}
              e^{-\hat p \lambda} \P(U''=0) + \sum_{u=1}^4 e^{-\hat p (1+u) \lambda } \P(U''=u) + \sum_{u\geq 5} e^{- \hat p (1+u) \lambda} \P(U''=u).
    \end{align}
    The first term is bounded by $e^{-\hat p \lambda } e^{-(1-\hat p) \lambda} = e^{-\lambda}.$
    Each summand in the second term is bounded by $e^{- 2\hat p \lambda} e^{- (1- \hat p) \lambda (d-5)/(d-1)}.$ Each summand in the third term is bounded by $e^{- 6 \hat p\lambda} \cdot 1.$
    
    Putting these bounds together, then factoring out $\lambda$ gives
    \begin{align}
             \E[ e^{-p^* - \hat p (1+ U'') \lambda }]  &\leq e^{- \f 17 }\left[e^{-\lambda} + 4 e^{- 2\hat p \lambda} e^{- (1- \hat p) \lambda (d-5)/(d-1)} +  e^{- 6\hat p  \lambda}  \right]\\
            &= e^{-\lambda }e^{- \f 17}\left[ 1 + 4 e^{ (1- 2 \hat p - (1- \hat p) (d-5)/(d-1) )  \lambda } + e^{(1 - 6 \hat p )\lambda}  \right].    \end{align}
Recall that for $p>1/6$, $\hat p > 1/5$ and thus $1- 6 \hat p \leq - 1/5$. Also, let $\delta>0$ be sufficiently small so that for $d$ large enough
$$1- 2\hat p  - (1- \hat p)\f{d-5}{d-1} \leq  -\delta. $$ For such $d$, we then have
\begin{align}
    \E[ e^{-p^* - \hat p (1+ U'') \lambda }] & \leq e^{-\lambda} e^{- \f 17 } \left[ 1 + 4 e^{-\delta \lambda} + e^{-\lambda/5} \right]. \label{eq:almost}
\end{align}
By taking $\lambda_0$ large enough, we get
\begin{align}
1 + 4 e^{-\delta \lambda_0} + e^{-\lambda_0/5} \leq e^{\frac{1}{56}}. \label{eq:done}
\end{align}
Applying \eqref{eq:done} to \eqref{eq:almost} gives \eqref{eq:U''} for all $\lambda \geq \lambda_0$ and $d$ large enough. 
\end{proof}


Now we are ready to prove our second theorem.

\begin{proof}[Proof of \thref{thm:lim}]
\thref{lem:transient} implies that $p_d' \geq 1/6$. For $p > 1/6$, \thref{lem:lb} ensures that $V_{\SFM(d,p)} \succeq \Poi(\lambda_0)$ with $\lambda_0$ from the proof of \thref{lem:U''} for all $d$ large enough.
It follows from \thref{prop:suff} that $\SFM(d,p)$ is recurrent. By \thref{lem:so}, $\NBFM(d,p)$ is recurrent. This gives \thref{thm:lim}.
\end{proof}

\begin{proof}[Proof of \thref{thm:limsup}]
    The result follows from \thref{thm:lim} and \thref{lem:so}.
\end{proof}

\section{Numerical simulations} \label{sec:sims}

We describe some numerical simulations to estimate $p_3$, $p_3'$, and $p_4'$. The plots in Figures~\ref{fig:fmd3},~\ref{fig:nbfmd3},~\ref{fig:nbfmd4} were created using SageMath~\cite{sagemath}, extending the code used in~\cite{hoffman2017recurrence}. Our code and a readme file are posted to the arXiv ancillary files for this article. The approach mirrors that of~\cite{hoffman2017recurrence} with the appropriate adaptations for working with $\FM(d,p)$ and $\NBFM(d,p)$.  To keep this section self-contained, we briefly review the method. 


An adaptation of $\FM(d,p)$ or $\NBFM(d,p)$ is to insert stunning fences at each depth $\ell$.  The role of these fences is to stun (i.e., temporarily freeze in place) frogs when they first reach level $\ell$.  Let $A_{d,p}(\ell)$ (resp. $A'_{d,p}(\ell)$) be the number of frogs in the $\FM(d,p)$ model (resp. $\NBFM(d,p)$) that reach level $\ell$ once all awake frogs have been stunned. As noted in~\cite{hoffman2017recurrence}, counting the number of root visits in direct simulations of either model is difficult to analyze due to rapid growth. Alternatively, we examine summability of 
\begin{equation}\label{eq:simsum}
    \sum_{\ell\geq 1}\hat{p}^\ell\E[A_{d,p}(\ell)],  
\end{equation}
recalling the notation from \eqref{eq:p*}, $\hat{p}=p/(1-p)$.  Notice that if 
\[N_k=N_k(d,p)\coloneqq\#\{\text{The root is visited between the $k$ and $k+1$th stunnings}\},\] then $\E[N_k]\gg \hat{p}^k\E[A_{d,p}(k)]$.  Hence the expected number of root visits is infinite when \eqref{eq:simsum} diverges.  This is strong evidence for the associated model being recurrent (for fixed $d, p$).  In particular \eqref{eq:simsum} diverges when
\[s_{d,p}(\ell) \coloneqq \ell\hat{p}^\ell\E[A_{d,p}(\ell)]\] is bounded below (and similarly for $s'_{d,p}(\ell)$ for non-backtracking).

Figures~\ref{fig:fmd3}, ~\ref{fig:nbfmd3},~\ref{fig:nbfmd4} show simulations of $s_{3,p}(\ell)$, $s'_{3,p}(\ell)$, and $s'_{4,p}(\ell)$ for levels $\ell\leq15$, for three different values of $p$ in each case.  The simulations were run with $1000$ iterations up to level $10$, and with $500$ iterations for levels $11$ through $15$. Computational running time increases swiftly with $\ell$. The simulations suggest the values at \eqref{eq:sims}.


\begin{figure}
\centering
    \includegraphics[width=1\textwidth]{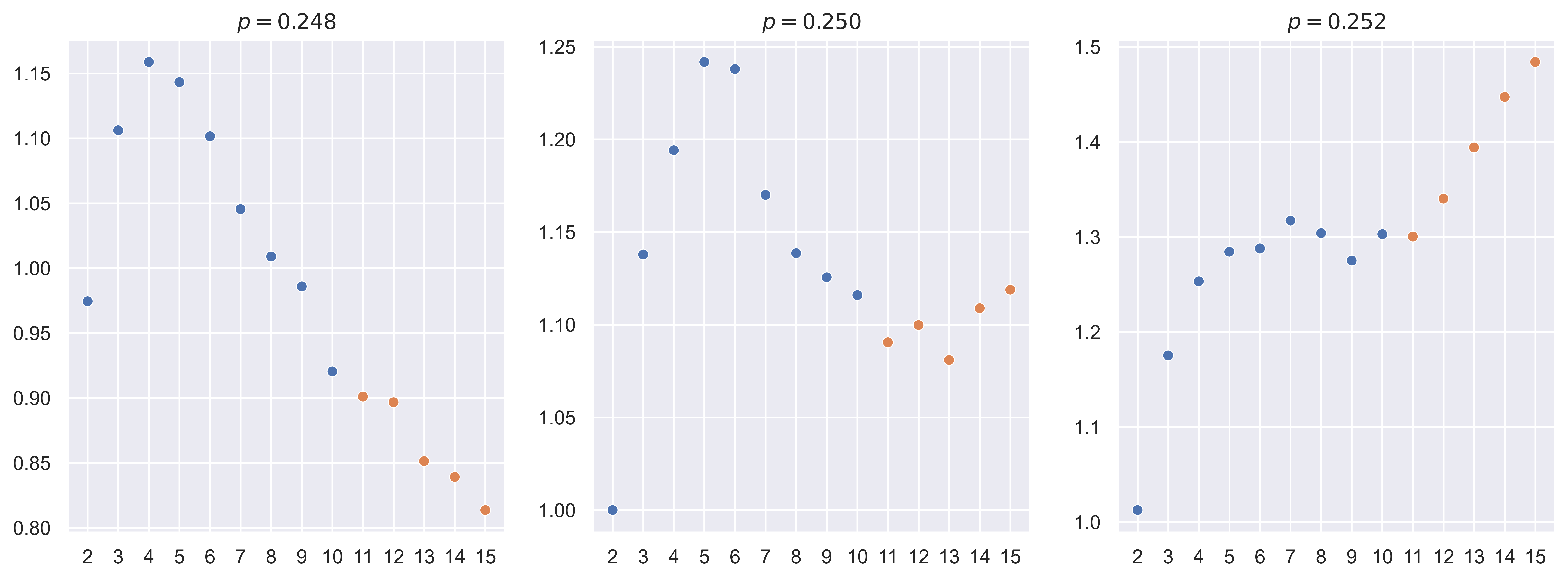}
    \caption{Simulated values of $s_{3,p}(\ell)$ for $\FM(3,p)$ for different values of $p$, for levels $\ell\leq 15$. The blue points represent simulations of $1000$ trials, and orange points represent simulations of $500$ trials.}\label{fig:fmd3}
\end{figure}
\begin{figure}
\centering
    \includegraphics[width=1\textwidth]{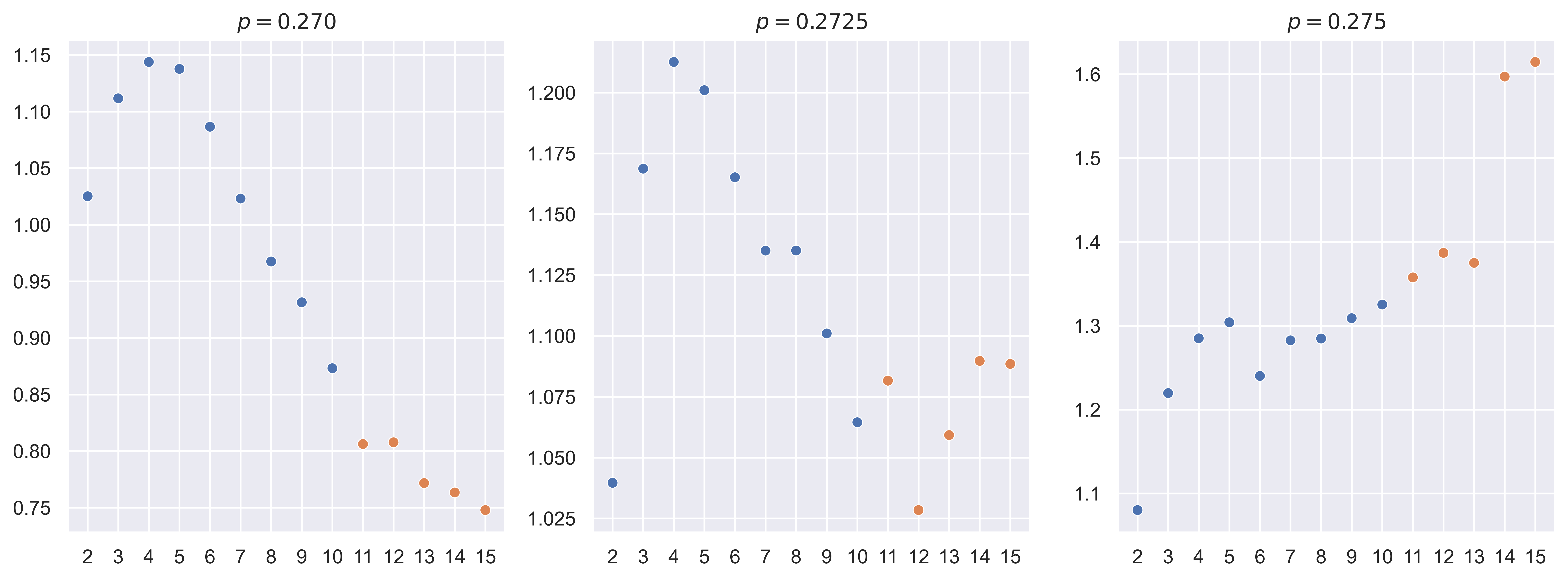}
    \caption{Simulated values of $s'_{3,p}(\ell)$ for $\NBFM(3,p)$ for different values of $p$, for levels $\ell\leq 15$. The blue points represent simulations of $1000$ trials, and orange points represent simulations of $500$ trials.}\label{fig:nbfmd3}
    \end{figure}    
\begin{figure}
\centering
    \includegraphics[width=1\textwidth]{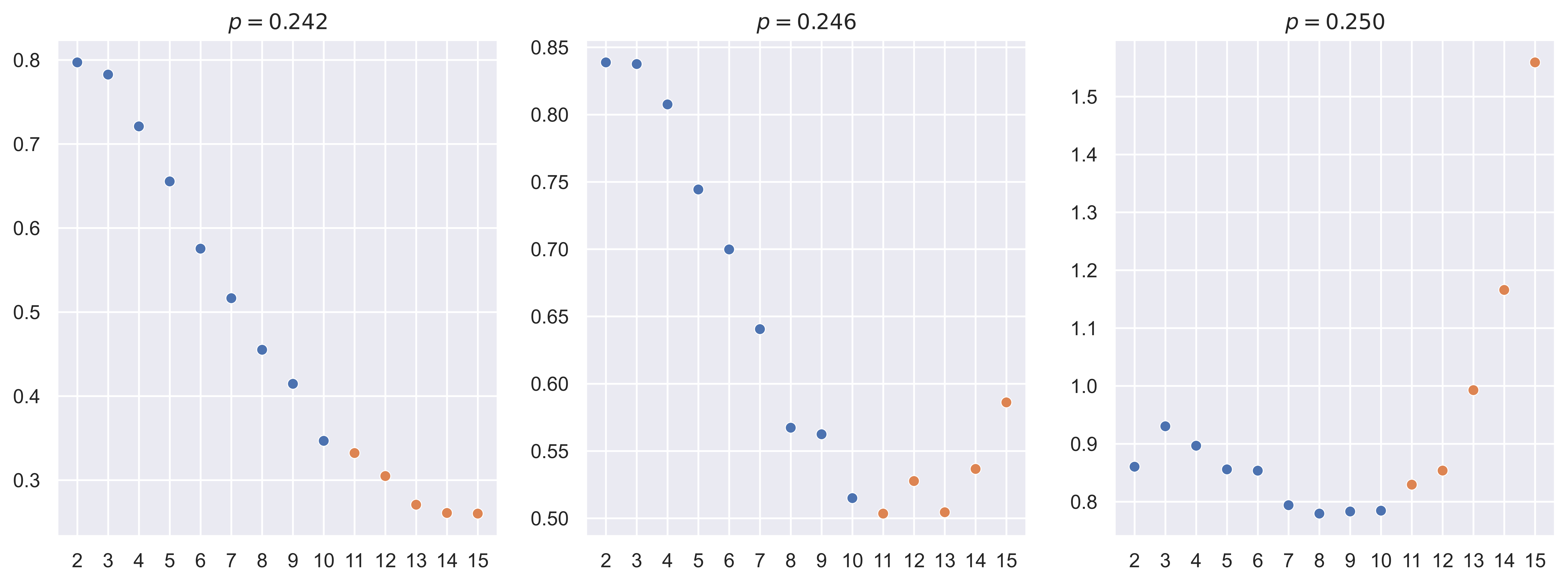}
    \caption{Simulated values of $s'_{4,p}(\ell)$ for $\NBFM(4,p)$ for different values of $p$, for levels $\ell\leq 15$. The blue points represent simulations of $1000$ trials, and orange points represent simulations of $500$ trials.} \label{fig:nbfmd4}
\end{figure}


\newpage 

\bibliographystyle{alpha}

\bibliography{frog}

\end{document}